\documentclass[11pt]{article}
\usepackage[margin=1in]{geometry} 
\usepackage{amssymb,amsfonts,amsmath,bbm,mathrsfs,stmaryrd}
\usepackage{mathtools}
\usepackage{xcolor}
\usepackage{url}

\usepackage{enumitem}

\setlist[enumerate,1]{label=\textup{(\alph*)}}
\setlist{nosep}

\usepackage[colorlinks,
             linkcolor=black!75!red,
             citecolor=blue,
             pdftitle={Top_gen},
             pdfauthor={Alexandru Chirvasitu, Frieder Ladisch, Pablo Soberon},
             pdfproducer={pdfLaTeX},
             pdfpagemode=None,
             bookmarksopen=true
             bookmarksnumbered=true]{hyperref}

\usepackage{tikz}
\usetikzlibrary{arrows,calc,decorations.pathreplacing,decorations.markings,intersections,shapes.geometric,through,fit,shapes.symbols,positioning,decorations.pathmorphing}

\usepackage{braket}

\usepackage[amsmath,thmmarks,hyperref]{ntheorem}
\usepackage{cleveref}

\creflabelformat{enumi}{#2#1#3}

\crefname{section}{Section}{Sections}
\crefformat{section}{#2Section~#1#3} 
\Crefformat{section}{#2Section~#1#3} 

\crefname{subsection}{\S}{\S\S}
\crefformat{subsection}{#2\S#1#3} 
\Crefformat{subsection}{#2\S#1#3} 

\theoremstyle{plain}

\newtheorem{lemma}{Lemma}[section]
\newtheorem{proposition}[lemma]{Proposition}
\newtheorem{corollary}[lemma]{Corollary}
\newtheorem{theorem}[lemma]{Theorem}

\newtheorem{question}[lemma]{Question}

\theoremstyle{nonumberplain}

\theoremstyle{plain}
\theorembodyfont{\upshape}
\theoremsymbol{\ensuremath{\blacklozenge}}

\newtheorem{definition}[lemma]{Definition}

\newtheorem{remark}[lemma]{Remark}

\crefname{definition}{definition}{definitions}
\crefformat{definition}{#2definition~#1#3} 
\Crefformat{definition}{#2Definition~#1#3} 

\crefname{ex}{example}{examples}
\crefformat{example}{#2example~#1#3} 
\Crefformat{example}{#2Example~#1#3} 

\crefname{remark}{remark}{remarks}
\crefformat{remark}{#2remark~#1#3} 
\Crefformat{remark}{#2Remark~#1#3} 

\crefname{convention}{convention}{conventions}
\crefformat{convention}{#2convention~#1#3} 
\Crefformat{convention}{#2Convention~#1#3}

\crefname{lemma}{lemma}{lemmas}
\crefformat{lemma}{#2lemma~#1#3} 
\Crefformat{lemma}{#2Lemma~#1#3} 

\crefname{proposition}{proposition}{propositions}
\crefformat{proposition}{#2proposition~#1#3} 
\Crefformat{proposition}{#2Proposition~#1#3} 

\crefname{corollary}{corollary}{corollaries}
\crefformat{corollary}{#2corollary~#1#3} 
\Crefformat{corollary}{#2Corollary~#1#3} 

\crefname{theorem}{theorem}{theorems}
\crefformat{theorem}{#2theorem~#1#3} 
\Crefformat{theorem}{#2Theorem~#1#3} 

\crefname{assumption}{assumption}{Assumptions}
\crefformat{assumption}{#2assumption~#1#3} 
\Crefformat{assumption}{#2Assumption~#1#3} 

\crefname{equation}{}{}
\crefformat{equation}{(#2#1#3)} 
\Crefformat{equation}{(#2#1#3)}

\theoremstyle{nonumberplain}
\theoremsymbol{\ensuremath{\blacksquare}}

\newtheorem{proof}{Proof}

\numberwithin{equation}{section}

\newcommand\bC{{\mathbb C}}

\newcommand\bF{{\mathbb F}}

\newcommand\bQ{{\mathbb Q}}
\newcommand\bR{{\mathbb R}}
\newcommand\bS{{\mathbb S}}

\newcommand\bZ{{\mathbb Z}}

\newcommand\cF{{\mathcal F}}

\newcommand\cP{{\mathcal P}}
\newcommand\cQ{{\mathcal Q}}

\DeclareMathOperator{\rank}{rank}
\DeclareMathOperator{\Hom}{Hom}
\DeclareMathOperator{\cpd}{cpd}
\DeclareMathOperator{\gcpd}{gcpd}
\DeclareMathOperator{\orth}{\mathbf{O}}  % orthogonal group
\DeclareMathOperator{\GL}{GL}
\DeclareMathOperator{\Aff}{Aff}   % affine group
\DeclareMathOperator{\conv}{conv}   % convex hull
\DeclareMathOperator{\skel}{skel}   % skeleton

\DeclarePairedDelimiter{\abs}{\lvert}{\rvert}

\newcommand{\qedhere}{\mbox{}\hfill\ensuremath{\blacksquare}}

\title{Non-commutative groups as prescribed polytopal symmetries}
\author{Alexandru Chirvasitu, Frieder Ladisch, and Pablo Sober\'on}

\begin{document}

\date{}

\newcommand{\Addresses}{{% additional braces for segregating \footnotesize
  \bigskip
  \footnotesize

  \textsc{Department of Mathematics, University at Buffalo, Buffalo,
    NY 14260-2900, USA}\par\nopagebreak \textit{E-mail address}:
  \texttt{achirvas@buffalo.edu}

  \medskip
  
  \textsc{Institut für Mathematik, Universität Rostock,
          18051 Rostock, 
          Germany}\par\nopagebreak \textit{E-mail address}:
     \texttt{frieder.ladisch@uni-rostock.de}
  
  \medskip
  
  \textsc{Department of Mathematics, Baruch College, City University of New York,
    New York, NY 10010, USA}\par\nopagebreak \textit{E-mail address}:
  \texttt{pablo.soberon-bravo@baruch.cuny.edu}

% %   \medskip
% %   
% %   \textsc{Department of Mathematics, institution,
% %     address}\par\nopagebreak \textit{E-mail address}:
% %   \texttt{??}
% % 

}}

\maketitle

\begin{abstract}

We study properties of the realizations of groups as the combinatorial automorphism group of a convex polytope.  We show that for any non-abelian group $G$ with a central involution there is a centrally symmetric polytope with  $G$ as its combinatorial automorphisms. We show that for each integer $n$, there are groups that cannot be realized as the combinatorial automorphisms of convex polytopes of dimension at most $n$.  We also give an optimal lower bound for the dimension of the realization of a group as the group of isometries that preserves a convex polytope.

\end{abstract}


\vspace{.5cm}


\tableofcontents

%%%%%%%%%%%%%%%%%%%%%%%%%%%%%%%%%%%%%%%%%%%%%%%%%%%%%%%%%%%%%%%%%%%%%%%%%%%%%%%%%%%%%%%%%%%%%%%%%%%%%%%%%%%%%%%%%%
%%%%%%%%%%%%%%%%%%%%%%%%%%%%%%%%%%%%%%%%%%%%%%%%%%%%%%%%%%%%%%%%%%%%%%%%%%%%%%%%%%%%%%%%%%%%%%%%%%%%%%%%%%%%%%%%%%
%\section*{Introduction}

%The goal will be to answer \cite[Open Questions 1, 2 and 3]{ssw} affirmatively. 

%%%%%%%%%%%%%%%%%%%%%%%%%%%%%%%%%%%%%%%%%%%%%%%%%%%%%%%%%%%%%%%%%%%%%%%%%%%%%%%%%%%%%%%%%%%%%%%%%%%%%%%%%%%%%%%%%%
\subsection*{Acknowledgements}

A.C. is grateful for partial support through NSF grant DMS-1801011.
P.S. is grateful for partial support through NSF grant DMS-1851420.
F.L. is grateful for support by DFG, Project SCHU1503/7-1.

%%%%%%%%%%%%%%%%%%%%%%%%%%%%%%%%%%%%%%%%%%%%%%%%%%%%%%%%%%%%%%%%%%%%%%%%%%%%%%%%%%%%%%%%%%%%%%%%%%%%%%%%%%%%%%%%%%
%%%%%%%%%%%%%%%%%%%%%%%%%%%%%%%%%%%%%%%%%%%%%%%%%%%%%%%%%%%%%%%%%%%%%%%%%%%%%%%%%%%%%%%%%%%%%%%%%%%%%%%%%%%%%%%%%%
\section*{Introduction}\label{se.intro}

Polyhedra and their symmetries have been a rich subject of study.  In this paper we are interested in the symmetries of convex polytopes, the convex hulls of finite sets of points in $\mathbb{R}^d$.  For a given polytope~$\cP$, we consider the following two symmetry groups: The \emph{geometric symmetry group} $G(\cP)$ consists of all (Euclidean) isometries of the ambient space that map $\cP$ onto itself.  The \emph{combinatorial symmetry group} $\Gamma(\cP)$ of $\cP$ consists of all the automorphisms of the face lattice of $\cP$. The group $\Gamma(\cP)$ can be identified with the group of all permutations of the vertices of $\cP$ which map faces of $\cP$ to faces of $\cP$.  It has recently been established that every finite group is isomorphic to the combinatorial automorphism group of a convex polytope $\cP$.  This result was first proved by Schulte and Williams \cite{sw} and later by Doignon \cite{Doignon:2018ch} with a simplified proof.  In Doignon's paper, it was shown that every finite group is the combinatorial automorphism group of a $0/1$-polytope.

Given a group $G$, there are many different convex polytopes whose symmetry group is exactly $G$.  This leads to a rich family of extremal problems: \textit{given a group $G$ and a parameter $\lambda( \cdot )$, determine the maximum/minimum value that $\lambda(P)$ can take, where $\cP$ ranges over the polytopes whose (geometric or combinatorial) symmetry group is exactly $G$.} In this note we are interested in finding polytopes $\cP$ of minimal dimension with prescribed symmetry group.  We give a lower bound for the dimension of a polytope having an elementary abelian group of given order as combinatorial automorphism group (\Cref{pr.simpl-sph}).  A consequence is that for every $d$, there are finite groups which can not be realized as combinatorial automorphism groups of convex polytopes of dimension less than or equal to $ d$.
It turns out that for elementary abelian groups, our lower bound is sharp (\Cref{th:cpd-sharp}). While our results also yield lower bounds for the dimension of a polytope with some other prescribed finite group as combinatorial automorphism group, these lower bounds are probably not optimal when the group is not elementary abelian.

We also show that finding a polytope of smallest dimension with a prescribed group as \emph{geometric} symmetry group reduces to a problem in the representation theory of finite groups.  Using this, we can compute \emph{exactly} the smallest dimension of a polytope having a given abelian group as geometric symmetry group (\Cref{th.bd}), and similarly for other classes of groups. For example, the symmetric group $S_n$ on $n$ letters is the geometric (and combinatorial) symmetry group of a regular simplex of dimension $n-1$, and this is in fact the smallest dimension of a polytope with geometric symmetry group $S_n$ (see \Cref{th.bd_sn_etal}).

%In particular, we can compute exactly the smallest dimension of a polytope having an elementary abelian group of given order as geometric symmetry group, but we only have a lower bound for the smallest dimension of a polytope with the same elementary abelian group as combinatorial symmetry group. Presumably, this lower bound is not optimal. 

The results described so far affirmatively answer two questions of Schulte, Sober\'on and Williams \cite[Open Questions~1 and 2]{ssw}.
We also consider the following problem suggested by Schulte, Sober\'on, and Williams: Instead of solving instances of the general extremal problems described above, one can look for polytopes with fixed symmetry groups that satisfy additional geometric properties.  This was explored in the cited paper~\cite{ssw}, where it was established that every finite abelian group of even order is the automorphism group of a \textit{centrally symmetric} polytope.  Moreover, the involution that corresponds to the central symmetry can be prescribed in advance.  We generalize this result to non-abelian groups with a central involution, thereby answering another question from \cite{ssw} (see \Cref{th.sym}).  The methods we use come from representation theory and geometric group theory.  Representation theory has been used previously to study convex polytopes, as can be seen in \cite{Guralnick:2006gk, Friese:2016ka, Baumeister:2009dfa} and the references therein.

Our main results answer all the open questions \cite[Open Questions 1, 2 and 3]{ssw} affirmatively.  In \Cref{sec.openqu}, we propose some new questions, which are suggested by our results.

%%%%%%%%%%%%%%%%%%%%%%%%%%%%%%%%%%%%%%%%%%%%%%%%%%%%%%%%%%%%%%%%%%%%%%%%%%%%%%%%%%%%%%%%%%%%%%%%%%%%%%%%%%%%%%%%%%
%%%%%%%%%%%%%%%%%%%%%%%%%%%%%%%%%%%%%%%%%%%%%%%%%%%%%%%%%%%%%%%%%%%%%%%%%%%%%%%%%%%%%%%%%%%%%%%%%%%%%%%%%%%%%%%%%%
\section{Preliminaries}

Given a polytope $\cP$, we denote by $\Gamma(\cP)$ its group of combinatorial automorphisms and by $G(\cP)$ its group of geometric symmetries.  One of our main tools we use is the following theorem.

\begin{theorem}[\cite{ssw}]\label{theorem-symmetrybreaks}
  Let $d\geq 3$, let $Q$ be a convex $d$-polytope with (combinatorial) automorphism group $\Gamma(Q)$, and let $\Gamma$ be a subgroup of $\Gamma(Q)$. Then there exists a convex $d$-polytope~$\mathcal{P}$ with the following properties:
\begin{enumerate} 
\item  $\Gamma(\mathcal{P})=\Gamma$.
\item  \label{it:tess} $\mathcal{P}$ is isomorphic (as an abstract polytope) to a face-to-face tessellation $\mathcal{T}$ of the $(d-1)$-sphere $\mathbb{S}^{d-1}$ by spherical convex $(d-1)$-polytopes.
\item $\skel_{d-2}(\mathcal{C}(Q))$ is a subcomplex of $\skel_{d-2}(\mathcal{P})$.
\item  \label{it:geom} If $\Gamma$ is a subgroup of the (geometric) symmetry group $G(Q)$ of $Q$, then the tessellation $\mathcal{T}$ on $\mathbb{S}^{d-1}$ in \labelcref{it:tess}  can be chosen in such a way that $G(\mathcal{T})=\Gamma=\Gamma(\mathcal{T})$.
\end{enumerate}
\end{theorem} 

In the statement above, $\skel_{d-2}(\cdot )$ stands for the $(d-2)$-dimensional skeleton of a polytope, and $\mathcal{C}(Q)$ for the barycentric subdivision of the boundary complex of $Q$. The proof of \Cref{theorem-symmetrybreaks} is done by taking a barycentric subdivision of $Q$ and then adding faces in order to break the undesired symmetries in $\cP$.  The method preserves any geometric symmetries of $Q$ that were present in $\Gamma$.  This method actually gives a stronger result.  Not only is the group of combinatorial automorphisms of $\cP$ equal to $\Gamma$, but the group of automorphisms of $\skel_{1}(\cP)$ (as a graph) is equal to $\Gamma$.  The reason for this is that the only tool that $\cite{ssw}$ uses to determine that extra symmetries have been broken is an analysis of the degrees of the vertices in $\skel_{1}(\cP)$.  It follows that in \Cref{th.sym} below, the combinatorial automorphism group~$\Gamma(\cP)$ can be replaced by the automorphism group of the $1$-skeleton (as a graph) of our polytope.

The following corollary will be useful a few times:
\begin{corollary}\label{cor.geosym-givengroup}
  Let $d\geq 3$ and let $\Gamma$ be a finite subgroup of the orthogonal group $\orth_d(\bR)$.  Then there exists a convex $d$-polytope~$\cP$ such that $G(\cP) = \Gamma(\cP) = \Gamma$.
\end{corollary}
\begin{proof}
  The group~$\Gamma$ acts on the unit sphere $\bS^{d-1}\subset \bR^d$.  If we take a sufficiently large union $F$ of finitely many $\Gamma$-orbits, then the convex hull $\cQ = \conv(F)$ will be a convex $d$-polytope.  By construction, $\Gamma \leq G(\cQ)$, so the result follows from~\Cref{theorem-symmetrybreaks}, \labelcref{it:geom}.
\end{proof}

A few times, we will use the following standard result from representation theory~\cite[Theorem~2.13]{huppCT}.
\begin{proposition}
\label{pr.inv_innprod}
Let $\rho\colon \Gamma\to \GL(d,\bR)$ be a representation of the finite group $\Gamma$.  Then there is a $\Gamma$-invariant inner product on $\bR^d$.  Equivalently, the representation $\rho$ is similar to a representation $\Gamma\to \orth_d(\bR)$.
\end{proposition}
\begin{proof}
  Let $(\cdot,\cdot)$ be an arbitrary inner product on $\bR^d$.  Then $[x,y]:= \sum_{g\in \Gamma} (\rho(g)x,\rho(g)y)$ yields a $\Gamma$-invariant inner product on $\bR^d$.  When $S=(s_1,\dotsc,s_d)$ is an orthonormal basis of $\bR^d$ with respect to $[\cdot,\cdot]$, then the corresponding matrices $S^{-1}\rho(g)S$ are orthogonal.
\end{proof}
A general consequence of this result is that it does not make much of a difference for the questions considered in this paper, whether we deal with the geometric symmetry group of polytopes, or the group of affine symmetries of polytopes (the affine maps sending a polytope to itself).

% % %%%%%%%%%%%%%%%%%%%%%%%%%%%%%%%%%%%%%%%%%%%%%%%%%%%%%%%%%%%%%%%%%%%%%%%%%%%%% 
% % %%%%%%%%%%%%%%%%%%%%%%%%%%%%%%%%%%%%%%%%%%%%%%%%%%%%%%%%%%%%%%%%%%%%%%%%%%%%% 
% % \section{Main results}\label{se.main}
% % 
% % We now turn to answering the open questions in \cite{ssw} noted above, tackling \cite[Open questions 1 and 2]{ssw} first.
% %

%%%%%%%%%%%%%%%%%%%%%%%%%%%%%%%%%%%%%%%%%%%%%%%%%%%%%%%%%%%%%%%%%%%%%%%%%%%%% 
%%%%%%%%%%%%%%%%%%%%%%%%%%%%%%%%%%%%%%%%%%%%%%%%%%%%%%%%%%%%%%%%%%%%%%%%%%%%% 
\section{Symmetric polytopes}\label{subse.sym}

Our first main result answers \cite[Open Question 3]{ssw} affirmatively.

\begin{theorem}\label{th.sym}
  Let $\Gamma$ be a finite group and $\sigma\in \Gamma$ a central order-two element. Then, there is a centrally symmetric polytope $\cP$ with
  \begin{equation*}
    G(\cP)=\Gamma(\cP)\cong \Gamma
  \end{equation*}
  and $\sigma$ acting as the central symmetry. 
\end{theorem}
\begin{proof}
  Let $G=\langle\sigma\rangle$ and $\rho:G\to \{\pm 1\}\subset \bR^{\times}$ the non-trivial representation, whose one-dimensional real carrier space we denote by $V$. Upon equipping the induced representation
  \begin{equation*}
    W:=\mathrm{Ind}_G^{\Gamma}V
  \end{equation*}
  with a $\Gamma$-invariant inner product
  (\Cref{pr.inv_innprod}), 
  we get an embedding of $\Gamma$ into the orthogonal group of $W$.
  Moreover, $\sigma$ acts as $-1$ on $W$ by construction.
  By \Cref{cor.geosym-givengroup}, 
  it follows that there is a polytope~$\cP$ such that 
  $G(\cP)=\Gamma(\cP)\cong \Gamma$. 
\end{proof}

% % In particular, it shows that any group with a central symmetry can be realized as the symmetry group of a centrally symmetric convex polytope.  In formal terms, we have the following corollary, whose validity was asked in.
% % 
% % \begin{corollary}
% %   Let $\Gamma$ be a finite group and $\sigma\in \Gamma$ a central order-two element $\sigma$. Then, there is a centrally symmetric polytope $\cP$ with
% %   \begin{equation*}
% %     G(\cP)=\Gamma(\cP)\cong \Gamma
% %   \end{equation*}
% %   and $\sigma$ acting as the central symmetry.  \qedhere
% % \end{corollary}
% %

%%%%%%%%%%%%%%%%%%%%%%%%%%%%%%%%%%%%%%%%%%%%%%%%%%%%%%%%%%%%%%%%%%%%%%%%%%%%%
%%%%%%%%%%%%%%%%%%%%%%%%%%%%%%%%%%%%%%%%%%%%%%%%%%%%%%%%%%%%%%%%%%%%%%%%%%%%%
\section{The convex polytope dimension}\label{subse.1-2}

Our next result shows that the ``convex polytope dimension of a group'' is a meaningful parameter.

\begin{definition}
  Let $G$ be a group.  We define $\cpd(G)$ (the \textit{convex polytope dimension of $G$}) to be the minimum integer $d$ such that there is a convex polytope $\mathcal{P}$ in $\mathbb{R}^d$ such that $\Gamma (\mathcal{P}) = G$.
\end{definition}

The results in \cite{sw} show that this parameter is well defined.  In this section we will prove that it can be arbitrarily large, and is therefore interesting to compute.  Note that when $A$ is a subgroup of $B$ and $\cpd(B)\geq 3$, then $\cpd(A) \le \cpd(B)$.  This is a direct consequence of \Cref{theorem-symmetrybreaks}.  On the other hand, 
the only groups with $\cpd(B)=2$ are the dihedral groups of order $2n\geq 6$, and the only groups with $\cpd(B)\leq 1$ are the groups of order $1$ and $2$. So when $B$ is a dihedral group and $A\leq B$ is cyclic of order $n\geq 3$, or the Klein four-group, then $\cpd(B)=2$, but $\cpd(A)=3$.  Aside from these exceptions, $\cpd(\cdot)$ is monotone.

\begin{remark}\label{re:act-on-sphr}
  As recalled in \cite[Section 2]{ssw}, to every convex polytope $P$ one can associate its {\it boundary complex}, consisting of its proper faces. This complex
  \begin{itemize}
  \item gives a tesselation of the boundary $\partial P$, which, topologically, is a sphere;
  \item has a barycentric subdivision giving a triangulation of said sphere, i.e., a {\it simplicial sphere};
  \item depends only on the combinatorial data of how the proper faces of $P$ are glued along facets.
  \end{itemize}
  It follows that the combinatorial symmetry group $\Gamma(\cP)$ of a polytope $\cP$ can be regarded as a group of {\it simplicial symmetries} of $\partial P$, i.e., automorphisms of $\partial P$ equipped with its simplicial complex structure. This renders plausible the relevance of the  result below.
\end{remark}

The proof of the following result requires, for a prime $p$ and an integer $n\ge -1$, the notion of a {\it cohomology $n$-sphere over $\bZ/p$}. Cohomology $n$-manifolds are introduced in \cite[Definition I.3.3]{borel-seminar} and are the central subject of study in said reference. They coincide with the objects introduced by Smith in \cite{smth-per2} to provide the appropriate context for his study of finite transformation groups.

For brevity, we refer to a cohomology $n$-manifold with the same homology as the sphere in $\bZ/p$ as a {\it cohomology $n$-sphere over $\bZ/p$}.

We do not recall the full definition of cohomology manifolds here, pausing only to note that
\begin{itemize}
\item simplicial $n$-spheres are cohomology $n$-spheres (over every $\bZ/p$);
\item cohomology $0$-spheres are just plain $0$-spheres, i.e. disjoint unions of two points (this follows for instance from \cite[5.3]{smth-per2}).
\end{itemize}

We will also be referring to the {\it link} $\mathrm{lk}_K(\Delta)$ of a simplex $\Delta$ in a simplicial complex $K$. This is usually defined via the {\it star} $\mathrm{st}_K(\Delta)$:
\begin{equation*}
  \mathrm{st}_K(\Delta) = \text{ subcomplex of }K\text{ generated by the simplices }\Delta'\text{ such that }\Delta\subseteq \Delta'\subseteq K. 
\end{equation*}
Then, the subcomplex of $\mathrm{st}_K(\Delta)$ consisting only of simplices disjoint from $\Delta$ is the link $\mathrm{lk}_K(\Delta)$. See e.g. \cite[discussion preceding Proposition 3.2.12]{fp-cell} or \cite[Definition 2.4.2]{maun-at}, though in the latter reference our `stars' are `nerves'. The introductory material in \cite[\S 2]{bry-pl} is another reference for these notions.

It follows for instance from \cite[Theorem X.6.9]{wild-man} that in a simplicial complex that is a cohomology $n$-manifold, the link on a $d$-simplex is a cohomology $(n-d-1)$-sphere. 

The following result is essentially \cite[second corollary, p.107]{smth}, but it was not obvious to us how it follows from the discussion preceding it, so we include a proof for completeness. See also \Cref{re:p2} for further discussion.

\begin{theorem}\label{pr.simpl-sph}
  Let $p$ be a prime, and let $\Gamma=(\bZ/p)^r$ be an elementary abelian $p$-group acting simplicially and effectively on a simplicial $n$-sphere. Then, we have
  \begin{itemize}
  \item $r\le n+1$ if $p=2$; 
  \item $2r\le n+1$ otherwise.
  \end{itemize}
\end{theorem}
\begin{proof}
It will be convenient, throughout the proof, to denote
  \begin{equation*}
    e=
    \begin{cases}
      1 &\text{ if }p=2\\
      \frac 12 &\text{ otherwise}.
    \end{cases}
  \end{equation*}
  We prove the claim that $r\le e(n+1)$ by induction on $n$ for an action on a simplicial cohomology $n$-sphere $\cP$.
  
  \vspace{.5cm}

  {\bf Base case $n=0$.} We have observed above that cohomology $0$-spheres are just plain $0$-spheres. so the result is obvious: $\bZ/2$ can act effectively by permuting the two points that constitute the $0$-sphere, while $\bZ/p$, $p>2$ cannot.

  \vspace{.5cm}

  {\bf Induction step.} We can assume $r\geq 2$. Since according to \cite[first corollary, p.107]{smth} a non-cyclic abelian group cannot act freely on a cohomology sphere, there must be a cyclic subgroup
  \begin{equation*}
    \bZ/p\cong \langle \sigma\rangle = G\le \Gamma
  \end{equation*}
  with non-empty fixed-point set $\cF$. It is a non-trivial result that the simplicial subcomplex $\cF\subset\cP$ must again be a cohomology $m$-sphere for some $0\le m\le n$:
  \begin{itemize}
  \item It is shown in \cite[\S IV.4.3]{borel-seminar} that $\cF$ has the same $(\bZ/p)$-homology as an $m$-sphere;
  \item \cite[Theorem V.2.2]{borel-seminar} proves that $\cF$ is an orientable cohomology manifold;
  \item \cite[\S I.3.4]{borel-seminar} then shows that its dimension as a complex (and as a cohomology manifold) must be $m$.
  \end{itemize}
  We choose $\sigma$ such that the dimension of $\cF$ is maximal, and consider the pointwise isotropy group $\Gamma_{\cF}$ of $\cF$. Because $\Gamma$ is elementary abelian we have a decomposition
  \begin{equation*}
    \Gamma=\Gamma_{\cF}\oplus \widetilde{\Gamma},
  \end{equation*}
  and $\widetilde{\Gamma}$ acts on $\cF$ (because the latter is the fixed-point set of $\Gamma_{\cF}$ and $\widetilde{\Gamma}$ commutes with the latter). We first claim that this action of $\widetilde{\Gamma}$ on the ($m$-dimensional) cohomology sphere $\cF$ is faithful.

  Indeed, were it not, there would be an element $\widetilde{\sigma}\in \widetilde{\Gamma}$ that fixes a set containing $\cF$. This could not be proper containment  (for that fixed-point set would then be a strictly higher-dimensional cohomology sphere thus contradicting the maximality of $\dim \cF$), and hence
  \begin{equation*}
    \cP^{\widetilde{\sigma}} = \cP^{\sigma} = \cF. 
  \end{equation*}
  But then \cite[Theorem]{smth} would imply that $\widetilde{\sigma}\in \Gamma=\langle\sigma\rangle$, i.e. a contradiction.
  
  But by induction, the faithfulness of the action of $\widetilde{\Gamma}$ on $\cF$ then gives
  \begin{equation}\label{eq:2}
    \rank~\widetilde{\Gamma}\le e(m+1). 
  \end{equation}
  On the other hand, the fact that $\Gamma_{\cF}$ itself acts effectively on all of $\cP$ but trivially on $\cF$ implies that the resulting action of $\Gamma_{\cF}$ on the link of some $m$-simplex in $\cF\subset \cP$ is effective. We observed in the above discussion that said link is a simplicial cohomology sphere of dimension $m'$ satisfying
  \begin{equation}\label{eq:3}
    (m+1)+(m'+1)=n+1. 
  \end{equation}
  The inductive step applied to {\it this} action gives us
  \begin{equation}\label{eq:4}
    \rank~\Gamma_{\cF}\le e(m'+1),
  \end{equation}
  and adding together \Cref{eq:2,eq:4} produces, via \Cref{eq:3}, precisely the desired inequality
  \begin{equation*}
    r=
    \rank~\Gamma_{\cF}+\rank~\widetilde{\Gamma}
    \le
    e(n+1). 
  \end{equation*}
  This finishes the proof. 
\end{proof}

\begin{remark}\label{re:p2}
  As mentioned above, \Cref{pr.simpl-sph} is essentially \cite[second corollary, p.107]{smth}. The latter, though, states only the $2r\le n+1$ inequality, which cannot be correct in full generality: $\bZ/2$, of rank $r=1$, acts on the zero-sphere by the obvious permutation; in that case the inequality would read $2\le 1$.

  There are hints in \cite{smth} that perhaps only odd primes are considered (see e.g. footnote 6 therein), but this does not seem to be stated explicitly.
\end{remark}

\begin{remark}
  The inductive argument in the proof of \Cref{pr.simpl-sph} is very much in the spirit of \cite[Lemmas 2.2 and 2.3]{mz}.  
\end{remark}

As a consequence we have the following corollary.

\begin{corollary}\label{cor.comb}
For each $n$ there is a finite group that cannot be realized as $\Gamma(\cP)$ for any polytope with $\dim \cP<n$.   
\end{corollary}
\begin{proof}
  Immediate from \Cref{pr.simpl-sph} and \Cref{re:act-on-sphr} above, to the effect that $\Gamma(\cP)$ acts effectively on the simplicial sphere $\partial \cP$.
\end{proof}

\Cref{cor.comb} gives an affirmative answer to \cite[Open Question 1]{ssw}. On the other hand, since finite groups of isometries can be regarded as acting effectively and simplicially on simplicial spheres, \Cref{pr.simpl-sph} also proves an affirmative answer to \cite[Open Question 2]{ssw}:

\begin{corollary}\label{cor.geom}
  For each $n$ there is a finite group that cannot be realized as the geometric symmetry group $G(\cP)$ for any polytope with $\dim \cP<n$.  \qedhere
\end{corollary}

In fact, the last corollary and the geometric version of \Cref{pr.simpl-sph} follow more swiftly from the observation that an abelian group with minimal number $r$ of generators can not be embedded into $\GL(d,\bC)$ for $d < r$ (see \Cref{th.bd} below, proof of lower bound).

% % 
% % As an example, consider $G= (\mathbb{Z}/2)^r$.  \Cref{pr.simpl-sph} shows that when $G$ acts as combinatorial automorphisms on a $d$-polytope, then $r \le 2^{d-1}$, and thus $\cpd(G) \ge 1 +\log_2 (r)$.  For an upper bound, consider $\mathcal{Q}$ a hypercube in $\mathbb{R}^r$.  Clearly $G$ is a subgroup of $\Gamma ( \mathcal{Q})$, so for $r\ge 3$ it can be refined to a polytope $\mathcal{P}$ in $\mathbb{R}^r$ such that $\Gamma (\mathcal{P}) = G$.  In other words, $\cpd(G) \le r$.  For $(\mathbb{Z}/p)^r$ we get similar bounds.  In this case $\mathcal{Q}$ is simply
% % 
% % \[ \mathcal{Q} = \underbrace{C_p \oplus \ldots \oplus C_p}_{r \mbox{ times}},\] where $C_p$ is a polygon of $p$ sides in $\mathbb{R}^2$ and $\oplus$ denotes the standard direct sum.
% % % originally: tensor product, but I think its direct sum? (Frieder)
% % The vertices of $\mathcal{Q}$ are the direct sums of $r$-tuples of vertices of $C_p$.  As $\mathcal{Q}$ is a polytope in $\mathbb{R}^{2r}$, as long as $r \ge 2$ we can use \Cref{theorem-symmetrybreaks} to obtain a polytope $\mathcal{P} \subset \mathbb{R}^{2r}$ such that $\Gamma(\mathcal{P}) = (\mathbb{Z}/p)^r$.  This proves that $1 + \log_2 (r) \le \cpd((\mathbb{Z}/p)^r) \le 2r$.  This leads to the problem of determining the exact value of $\cpd((\bZ/p)^r)$.
% % 
% % The gap shown above does not appear if we ask for geometric symmetries, as we will show in \Cref{subse.dim}.  
% %
%%%%%%%%%%%%%%%%%%%%%%%%%%%%%%%%%%%%%%%%%%%%%%%%%%%%%%%%%%%%%%%%%%%%%%%%%%%%% 

%%%%%%%%%%%%%%%%%%%%%%%%%%%%%%%%%%%%%%%%%%%%%%%%%%%%%%%%%%%%%%%%%%%%%%%%%%%%%
\section{The geometric convex polytope dimension}\label{subse.dim}

\begin{definition}
  Let $\Gamma$ be a group.  Let $\gcpd(\Gamma)$ (the \textit{geometric convex polytope dimension of $\Gamma$}) be the minimum integer $d$ such that there is a convex polytope $\mathcal{P}$ in $\mathbb{R}^d$ such that $G (\mathcal{P}) = \Gamma$.
\end{definition}

\begin{proposition}\label{pr.cpd-leq-gcpd}
  When $\gcpd(\Gamma) \geq 3$, then $\cpd(\Gamma) \leq \gcpd(\Gamma)$.
\end{proposition}
\begin{proof}
  This follows from \Cref{theorem-symmetrybreaks}, Part~(d).
\end{proof}

\Cref{fig:c4polygones} indicates two constructions of polygons with geometric symmetry group the cyclic group of order $4$.  The same constructions apply to all cyclic groups $\bZ/n$ of order $n \geq3$, so $\gcpd(\bZ/n)=2$.  On the other hand, the combinatorial automorphism group of a polygon is a dihedral group, so $\cpd(\bZ/n)\geq 3$.  By \Cref{theorem-symmetrybreaks} or otherwise, we have $\cpd(\bZ/n) = 3$. Similarly, we have $\gcpd((\bZ/2)^2) = 2$ (take a rectangle), but $\cpd( (\bZ/2)^2) = 3$. It follows that the conclusion of \Cref{pr.cpd-leq-gcpd} is wrong only for cyclic groups of order $\geq 3$, and the Klein four-group $(\bZ/2)^2$.

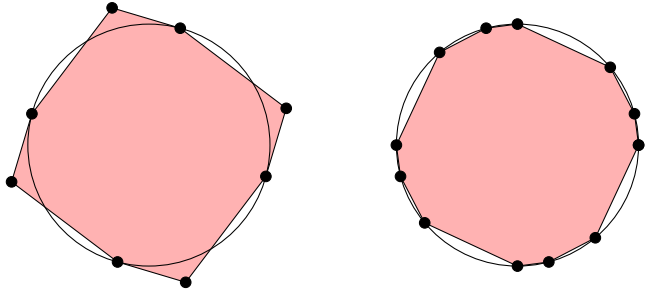
\begin{figure}
\centering
  \begin{tikzpicture}[scale=0.7]
   \tikzset{pnts/.style={circle, draw, fill=black,
                          inner sep=0pt, minimum width=4pt}}      \newcommand{\startangle}{15}
      \newcommand{\rada}{2.7}
      \newcommand{\radb}{2.3}
      
      \foreach \i in {0,1,2,3}{
         \coordinate (A\i) at (90*\i + \startangle:\rada);
         \coordinate (B\i) at (90*\i - \startangle:\radb);
      }
      
      \filldraw[fill=red!30]
        (A0)  node[pnts] {}
        \foreach \i in {1,2,3}{
          -- (B\i) 
              node[pnts] {}
          -- (A\i) 
              node[pnts] {}
        }
        -- (B0) 
           node[pnts]{}
        -- cycle;
      \node [draw, circle through=(B0)] at (0,0) {};
      
      \begin{scope}[xshift=7cm]
         \newcommand{\angleb}{15}
         \newcommand{\anglec}{40}
         \newcommand{\radius}{2.3}
         
         \foreach \i in {0,1,2,3}{
             \coordinate (A\i) at (90*\i:\radius);
             \coordinate (B\i) at (90*\i + \angleb:\radius);
             \coordinate (C\i) at (90*\i + \anglec:\radius);
         }
         
         \filldraw[fill=red!30]    
              (A0) node[pnts] {}
               \foreach \i in {0,1,2,3}{
                   -- (A\i)  node[pnts] {} 
                   -- (B\i)  node[pnts] {}
                   -- (C\i)  node[pnts] {}         
                   }
                   -- cycle;
         \node [draw, circle through=(B0)] at (0,0) {};
      \end{scope}
    \end{tikzpicture}
    \caption{Two polygons with $\bZ/4$ as geometric symmetry group}
    \label{fig:c4polygones}
\end{figure}

For geometric symmetries, we have the following alternative version of \Cref{theorem-symmetrybreaks}, which also works in dimension~$d\leq 2$.  The gist of this result is that we need to add at most one orbit of vertices to break undesired symmetries.  The proof is adapted from an argument by Isaacs~\cite{Isaacs77}.

\begin{proposition}\label{pr.geosymbreak}
  Let $\cQ \subset \bR^d$ be a finite, convex $d$-polytope and $\Gamma$ a subgroup of $G(\cQ)$.  Then there exists an orbit $X$ of $\Gamma$ on $\bR^d$ such that for the polytope $\cP = \conv(\cQ \cup X)$, we have $\Gamma=G(\cP)$.
\end{proposition}
\begin{proof}
  Let $t$ be the translation that sends the barycenter of $\cQ$ to the origin.  Then $G(t(\cQ))=tG(\cQ)t^{-1}$ and $t\Gamma t^{-1}$ is a subgroup of $G(t(\cQ))$.  So after replacing $\cQ$ by a translated copy, we may assume without loss of generality that the barycenter of $\cQ$ is the origin.  In this case, we have $G(\cQ) \subset \orth_d(\bR)$.  Let $Y$ be the vertex set of $\cQ$.  For every $1 \neq g\in G(\cQ)$, we have that $\ker(g-1)$ (the eigenspace of $g$ associated with the eigenvalue~$1$) is a proper subspace of $\bR^d$.  As $\bR^d$ is not the union of finitely many proper subspaces, we can find a vector $x$ such that no non-identity element of $G(\cQ)$ fixes $x$ and such that $x$ is not contained in a proper subspace which is spanned by some subset of $Y$.  Let $X = \Gamma x$ be the $\Gamma$-orbit of $x$.  For $\lambda>0$, set $\cP_{\lambda}= \conv(Y\cup \lambda X)$.  We will show that there is some $\lambda >0$ such that $G(\cP_{\lambda}) = \Gamma$.
   
  Consider the ray $\{\lambda x \colon \lambda >0\}$ spanned by $x$.  As $x$ is not contained in a proper subspace spanned by a subset of $Y$, this ray meets the boundary of $\cQ$ in the interior of a facet.  Thus there is a nonempty interval $I$ of positive real numbers such that the vertex set of $\cP_{\lambda}= \conv(Y\cup \lambda X)$ is exactly $Y\cup \lambda X$ for all $\lambda\in I$.
   
  Set $s = (1/\lvert G \rvert) \sum_{g\in \Gamma} gx$.  The barycenter of $\cP_{\lambda}$ is $\lambda s$.  For any $y\in Y$, there are at most two $\lambda$'s that solve the quadratic equation $\abs{ y-\lambda s}^2 = \abs{ \lambda x - \lambda s}^2$.  As $Y$ is finite, there are infinitely many $\lambda \in I$ such that $\abs{y-\lambda s} \neq \abs{ \lambda x - \lambda s}$ for all $y\in Y$.  We claim that $G(\cP_{\lambda}) = \Gamma$ for these $\lambda$.  By construction, $\Gamma \leq G(\cP_{\lambda})$ for any $\lambda$.  As $Y\cup \lambda X$ is the vertex set of $\cP_{\lambda}$, the isometry group $G(\cP_{\lambda})$ is the set of isometries mapping $Y \cup \lambda X$ to itself.  By the choice of $\lambda$, we have that $G(\cP_{\lambda})$ stabilizes $\lambda X$, and thus also $Y$.  Therefore, $G(\cP_{\lambda}) \subseteq G(\cQ)$.  By the choice of $x$, no non-identity element of $G(\cQ)$ fixes $x$.  Thus $\abs{G(\cP_{\lambda})} = \abs{G(\cP_{\lambda})x} \leq \abs{X} \leq \abs{\Gamma} $.  Since $\Gamma\subseteq G(\cP_{\lambda})$ we have $G(\cP_{\lambda}) = \Gamma$.
\end{proof}

Isaacs also showed that when some finite group $\Gamma \subset \orth_d(\bR)$ acts \emph{absolutely irreducibly} on $\bR^d$, then $\Gamma$ is the geometric symmetry group of a vertex-transitive polytope.  This can be generalized to some groups not necessarily acting absolutely irreducibly \cite[Corollary~5.8]{FrieseLadisch18}, but not to arbitrary subgroups of $\orth_d(\bR)$.

\begin{corollary}\label{cor.gcpdlin}
  Let $\Gamma$ be a finite group.  Then $\gcpd(\Gamma)$ equals the smallest $d$ such that $\Gamma$ embeds into $\GL(d,\bR)$.
\end{corollary}
\begin{proof}
  When $\Gamma\cong G(\cP)$ for some $d$-polytope~$\cP$, then $\Gamma$ embeds into $\orth_d(\bR) \subset \GL(d,\bR)$.
   
  Conversely, suppose $\Gamma$ embeds into $\GL(d,\bR)$.  By \Cref{pr.inv_innprod}, it follows that $\Gamma$ is isomorphic to a subgroup~$\widetilde{\Gamma}$ of $\orth_d(\bR)$.  As in the proof of \Cref{cor.geosym-givengroup}, we can find a polytope $\cQ$ such that $G(\cQ)$ contains the given subgroup $\widetilde{\Gamma}$ isomorphic to $\Gamma$.  Then by \Cref{pr.geosymbreak} (or \Cref{cor.geosym-givengroup} when $d\geq 3$), $\widetilde{\Gamma}$ is the isometry group of some $d$-polytope.
\end{proof}

\begin{remark}
  When $d$ is minimal such that the given group $\Gamma$ embeds into $\GL(d,\bR)$, then one can actually show that $\bR^d$ can be generated by a single $\Gamma$-orbit on $\bR^d$.  It follows that there is a $d$-polytope~$\cP$ with $\Gamma\cong G(\cP)$ and such that $\Gamma$ has at most two orbits on the set of vertices of $\cP$.  On the other hand, not every group is isomorphic to the geometric automorphism group of a vertex-transitive polytope~\cite{Babai77}.
\end{remark}

As an example, we can compute $\gcpd$ for abelian groups.

\begin{theorem}\label{th.bd}
  Let
  \begin{equation}\label{eq:1}
    \Gamma=\bZ/n_1\times\cdots\times \bZ/n_r
  \end{equation}
  be an abelian group, where
  \begin{equation*}
    n_1 \mid n_2 \mid \cdots \mid n_r
  \end{equation*}
  are the invariant factors of $\Gamma$. Let $s$ be the number of factors $n_i$ such that $n_i=2$ and let $t$ the number of factors $n_i$ such that $n_i>2$ (so $r=s+t$).  Then $\gcpd(\Gamma) =s+2t$, that is, the minimal dimension of a polytope $\cP$ such that $G(\cP)\cong \Gamma$ is $s+2t$.
\end{theorem}
\begin{proof}
  By \Cref{cor.gcpdlin}, we need to show that $d=s+2t$ is the minimum dimension of a faithful linear representation $\rho\colon \Gamma \to \GL(d,\bR)$.  The fact that this is a sharp bound entails two inequalities, which we prove separately.

  \vspace{.3cm}
  
  {\bf Lower bound.}  Suppose the linear representation $\rho\colon \Gamma\to \GL(d,\bR)$ is faithful.  So when decomposing the complexification of $\rho$ (denoted by the same symbol, for brevity) as a sum of irreducible (hence one-dimensional) characters, the summands $\chi_i$ must generate the Pontryagin dual group
  \begin{equation*}
    \widehat{\Gamma}:=\Hom(\Gamma,\bS^1)\cong \Gamma. 
  \end{equation*}
  It follows from this that there is, among the $\chi_i$, 
  a minimal set of generators $\chi_1$ up to $\chi_h$,
  so that $r \leq h \leq d$. 
  Minimality implies that no two $\chi_i$ and $\chi_j$, 
  $1\le i\ne j\le h$ can be mutually conjugate (since conjugation means taking the inverse in $\widehat{\Gamma}$).  

All self-conjugate $\chi_i=\overline{\chi_i}$ are trivial on the $t$-factor abelian group $2\Gamma$, so among the $h$ $\chi_i$ we must have at least $t$ non-self-conjugate characters. But because $\rho$ is a real representation, the $t$ conjugates $\overline{\chi_i}$ must be among the summands of $\rho$ as well, meaning that 
\begin{equation*}
  \dim \rho\ge h+t \ge r+t=s+2t,
\end{equation*}
as claimed. 

  \vspace{.3cm}

  {\bf Upper bound.} Fix an isomorphism
  \begin{equation*}
    \widehat{\Gamma}\cong \Gamma\cong \bZ/n_1\times\cdots\times \bZ/n_r
  \end{equation*}
  and select a set $\chi_i$, $1\le i\le r$ of generators for the $r$ factors. The $\chi_i$ of order two ($s$ of them, in the notation of the statement) are realizable over the real numbers, while the rest can be regarded as representations on $\bR^2\cong \bC$. Summing the characters we thus obtain an orthogonal representation of $\Gamma$ on $\bR^{s+2t}$, which is faithful because its simple constituents generate the dual group.
\end{proof}

We have similar sharp bounds in the centrally symmetric case:

\begin{theorem}\label{th.bd-sym}
  Let $\Gamma$ be a finite abelian group as in \Cref{eq:1}, $\sigma\in \Gamma$ a non-trivial involution, and $s$ and $t$ as in the statement of \Cref{th.bd}. When $\sigma$ is not a square in $\Gamma$, then the conclusion of \Cref{th.bd} holds for polytopes equipped with $\Gamma$-actions with $\sigma\mapsto -1$.  When $\sigma $ is a square, then the minimal dimension of a polytope $\cP$ with geometric symmetry group $\Gamma$ such that $\sigma$ acts as $-1$ is $2r=2(s+t)$.
\end{theorem}
\begin{proof}
  In any case, the fact that $\dim\cP\ge s+2t$ follows from \Cref{th.bd}, since the actions with $\sigma\mapsto -1$ form a subclass of isometric actions.
  
  When $\sigma$ is not a square in $\Gamma$, then we can write $\Gamma = \langle \sigma \rangle \times \Gamma_0$.  Thus there is a real character $\chi$ taking value $-1$ on $\sigma$.  Now choose generators $\chi_1=\chi$, $\chi_2$, $\dotsc$, $\chi_r$ as before. By replacing $\chi_i $ by $\chi\chi_i$ if necessary, we can ensure that $\chi_i(\sigma)=-1$ for all $i$.  Thus again the lower bound is achieved.
  
  Now suppose that $\sigma$ is a square in $\Gamma$. Then for every irreducible character $\chi$ taking value $-1$ on $\sigma$ we necessarily have $\chi\neq \overline{\chi}$. Thus in this case a faithful representation $\rho\colon \Gamma\to \GL(d,\bR)$ with $\rho(\sigma)=-I$ has at least $2r=2(s+t)$ summands.  Conversely, we can find a generating set $\chi_1$, $\dotsc$, $\chi_r$ with $\chi_i(\sigma)=-1$ for all $i$ (as above), so the bound is achieved.
\end{proof}

Going back to $\cpd$, we now know enough to determine it precisely for elementary abelian $p$-groups. 

\begin{theorem}\label{th:cpd-sharp}
  Let $p$ be a prime. For $\Gamma=(\bZ/p)^r$ we have $\cpd(\Gamma)=\gcpd(\Gamma)$ except in the following cases:
  \begin{itemize}
  \item $r=1$ and $p>2$, when $\cpd(\Gamma)=3$ and $\gcpd(\Gamma)=2$;
  \item $r=2$ and $p=2$, when again $\cpd(\Gamma)=3$ and $\gcpd(\Gamma)=2$.
  \end{itemize}
\end{theorem}
\begin{proof}
  By \Cref{pr.simpl-sph}, the dimension $n+1$ of a polytope $\cP$ on whose boundary spherical simplicial complex $\partial\cP$ the group $\Gamma$ can act effectively is at least the value of $\gcpd$ as computed in \Cref{th.bd} (namely $r$ when $p=2$ and $2r$ otherwise). In short, for elementary abelian $p$-groups $\Gamma$ we always have
  \begin{equation}\label{eq:5}
    \cpd(\Gamma)\ge \gcpd(\Gamma). 
  \end{equation}
  The two exceptions in the statement were noted and analyzed in the discussion following \Cref{pr.cpd-leq-gcpd}. As for the other cases, \Cref{th.bd} says that we then have $\gcpd(\Gamma)\ge 3$ and hence $\cpd(\Gamma)\le \gcpd(\Gamma)$ by \Cref{pr.cpd-leq-gcpd}. Combining this with \Cref{eq:5} gives the conclusion.
\end{proof}

In principle, \Cref{cor.gcpdlin} allows us to compute $\gcpd(\Gamma)$ from the character table of $\Gamma$.  A few more (elementary) examples which show that $\gcpd(\Gamma)$ can be arbitrarily large:
\begin{theorem}\label{th.bd_sn_etal}
  \begin{enumerate}
  \item $\gcpd(S_n) = n-1$, where $S_n$ is the symmetric group on $n$ letters.
  \item Let $\Gamma$ be an extraspecial $p$-group of order $p^{2k+1}$, $p$ odd.  Then $\gcpd(\Gamma) = 2 p^k$.
  \item Let $\Gamma = \Aff(\bF_q)$ be the affine group of degree~$1$ over the field $\bF_q$ with $q$ elements, that is, the set of all maps $\bF_q\to\bF_q$ of the form $x\mapsto ax+b$, $a$, $b\in \bF_q$, $a\neq 0$.  Then $\gcpd(\Gamma)= q-1$.
  \end{enumerate}
\end{theorem}
\begin{proof}
  \begin{enumerate}
  \item $S_n$ is the geometric (and combinatorial) automorphism group of a regular $(n-1)$-simplex.  The one-dimensional characters of $S_n$ have the alternating group in the kernel.  For $n\neq 4$, the symmetric group $S_n$ has no irreducible representations of degree between $1$ and $n-1$ \cite[pp.~466--468]{BurnsideTGFO}, and for $n=4$, the irreducible representation of degree~$2$ is not faithful.  Thus $\gcpd(S_n)=n-1$.
  \item The irreducible character degrees of an extraspecial $p$-group are $1$ and $p^k$ \cite[7.6(b)]{huppCT}.  The one-dimensional characters have the commutator subgroup in the kernel, so are not faithful.  Every character $\chi$ of degree $p^k$ is faithful, but $\chi\neq \overline{\chi}$ as $p$ is odd.  The representation $\Gamma\to \GL(p^k,\bC)$ affording $\chi$ can be viewed as a representation $\Gamma \to \GL(2p^k,\bR)$, so $\gcpd(\Gamma)=2p^k$.
  \item The affine group is not abelian and has only one nonlinear character, of degree~$q-1$ \cite[7.9(c)]{huppCT}.  This character is afforded by a representation over $\bR$ (in fact over $\bQ$), so $\gcpd(\Gamma)=q-1$.
  \end{enumerate}
\end{proof}

%%%%%%%%%%%%%%%%%%%%%%%%%%%%%%%%%%%%%%%%%%%%%%%%%%%%%%%%%%%%%%%%%%%%%%%%%%%%%

%%%%%%%%%%%%%%%%%%%%%%%%%%%%%%%%%%%%%%%%%%%%%%%%%%%%%%%%%%%%%%%%%%%%%%%%%%%%%

\section{Some open questions}
\label{sec.openqu}

By the results in the last section, we can (in principle) compute the geometric convex polytope dimension $\gcpd(\Gamma)$ of a group using representation theory.  For the convex polytope dimension, the situation is less clear. We have sharp bounds for elementary abelian $p$-groups by \Cref{th:cpd-sharp}, but the analogous question arises naturally for other families:

\begin{question}\label{qu.exact-cpd}
  What is the exact value of $\cpd(S_n)$ or $\cpd(\Aff(\bF_q))$ ?
\end{question}

Notice that for symmetric groups we do at least know that $\cpd(S_n)\to \infty$ for $n\to \infty$, as $S_n$ contains a subgroup $(\bZ/2)^{\lfloor n/2 \rfloor}$. For $\cpd(\Aff(\bF_p))$, $p$ prime, we do not known even this, since abelian subgroups of $\Aff(\bF_p)$ are cyclic.

By \Cref{pr.cpd-leq-gcpd}, we have $\cpd(\Gamma) \leq \gcpd(\Gamma)$ for all groups $\Gamma$ except cyclic groups of order $\geq 3$ and the Klein four-group.  This motivates the following question:

\begin{question}\label{qu.cpd-gcpd}
  Is there a group $\Gamma$ such that $\cpd(\Gamma) < \gcpd(\Gamma)$ ?
\end{question}

When such a group $\Gamma$ exists, then there must be a polytope~$\cP$ of dimension $d=\cpd(\Gamma)$, such that $\Gamma\cong \Gamma(\cP)$ is not isomorphic to a subgroup of $\GL(d,\bR)$.  Conversely, when $\cP$ is a convex $d$-polytope~$\cP$ such that $\Gamma=\Gamma(\cP)$ is not isomorphic to a subgroup of $\GL(d,\bR)$, then $\cpd(\Gamma) \leq d < \gcpd(\Gamma)$, as $\GL(k,\bR) \subset \GL(d,\bR)$ for $k < d$.

Of course, a negative answer to \Cref{qu.cpd-gcpd} would also settle \Cref{qu.exact-cpd}.

Bokowski, Ewald and Kleinschmidt~\cite{BokowskiEK84} found the first example of a polytope~$\cP$ (in dimension~$4$) which has a combinatorial automorphism~$\varphi$ that can not be realized geometrically in the sense that the polytope has no geometric realization that admits $\varphi$ as geometric automorphism. Mani~\cite{Mani71} showed that in dimension~$3$, for every convex polytope there is a combinatorially equivalent polytope, such that all its combinatorial automorphisms come from geometric automorphisms.  It follows that when $\gcpd(\Gamma) > 3$, then also $\cpd(\Gamma) >3$.  Not much else seems to be known.

If additional geometric structure is required on the realization of $G$ as a group of combinatorial symmetries of a polytope $\cP$, it can imply restrictions on the dimension of $\cP$.  Interestingly, such conditions are more diverse than just lower bounds on the dimension.  For example, if $\sigma$ is a central involution and there exists an element $\gamma$ such that $\gamma^2 = \sigma$, then for $\sigma$ to act as a central symmetry we need the ambient dimension to be even.  This is implied by a simple degree argument \cite[Proof of Theorem~4.2]{ssw}. Therefore, if we impose the way that a subgroup $A \subset G$ has to act on our polytope $\cP$, it can heavily restrict on which dimensions admit such a polytope.

%%%%%%%%%%%%%%%%%%%%%%%%%%%%%%%%%%%%%%%%%%%%%%%%%%%%%%%%%%%%%%%%%%%%%%%%%%%%%%%%%%%%%%%%%%%%%%%%%%%%%%%%%%%%%%%%%%
%%%%%%%%%%%%%%%%%%%%%%%%%%%%%%%%%%%%%%%%%%%%%%%%%%%%%%%%%%%%%%%%%%%%%%%%%%%%%%%%%%%%%%%%%%%%%%%%%%%%%%%%%%%%%%%%%%

\addcontentsline{toc}{section}{References}
\providecommand{\bysame}{\leavevmode\hbox to3em{\hrulefill}\thinspace}
\providecommand{\MR}{\relax\ifhmode\unskip\space\fi MR }
\providecommand{\MRhref}[2]{%
  \href{http://www.ams.org/mathscinet-getitem?mr=#1}{#2}
}
\providecommand{\href}[2]{#2}

\Addresses

\end{document}